\documentclass[a4paper,12pt]{amsart}
\usepackage[colorlinks,linkcolor=blue,citecolor=blue]{hyperref}
\usepackage{latexsym, amssymb, amsmath, amscd, amsthm, mathrsfs, bbm}
\usepackage[all, knot]{xy}
\xyoption{arc}

\usepackage{anysize}\marginsize{22mm}{22mm}{25mm}{25mm}
\allowdisplaybreaks[4]

\def \k {\mathbbm{k}}

\numberwithin{equation}{section}
\numberwithin{table}{section}
\numberwithin{equation}{section}
\newtheorem{theorem}{Theorem}[section]

\newtheorem{proposition}[theorem]{Proposition}
\newtheorem{corollary}[theorem]{Corollary}

\newtheorem{remark}[theorem]{Remark}

\title{Harrison center and products of sums of powers}
\thanks{Supported by NSFC 11971181 and 12131015.}



\author{Hua-Lin Huang, Lili Liao, Huajun Lu and Chi Zhang}

\address{School of Mathematical Sciences, Huaqiao University, Quanzhou 362021, China}
\email{hualin.huang@hqu.edu.cn}

\address{School of Mathematical Sciences, Huaqiao University, Quanzhou 362021, China}
\email{lili.liao@stu.hqu.edu.cn}

\address{School of Mathematical Sciences, Huaqiao University, Quanzhou 362021, China}
\email{huajun@hqu.edu.cn}

\address{Department of Mathematics, Northeastern University, Shenyang 110819, China}
\email{zhangchi@mail.neu.edu.cn}

\date{}                                           

\begin{document}

\begin{abstract}
This paper is mainly concerned with identities like \[ (x_1^d + x_2^d + \cdots + x_r^d) (y_1^d + y_2^d + \cdots y_n^d) = z_1^d + z_2^d + \cdots + z_n^d \] where $d>2,$ $x=(x_1, x_2, \dots, x_r)$ and $y=(y_1, y_2, \dots, y_n)$ are systems of indeterminates and each $z_k$ is a linear form in $y$ with coefficients in the rational function field $\k (x)$ over any field $\k$ of characteristic $0$ or greater than $d.$ These identities are higher degree analogue of the well-known composition formulas of sums of squares of Hurwitz, Radon and Pfister. We show that such composition identities of sums of powers of degree at least $3$ are trivial, i.e., if $d>2,$ then $r=1.$ Our proof is simple and elementary, in which the crux is Harrison's center theory of homogeneous polynomials.
\end{abstract}

\maketitle

\section{Introduction}

Throughout the paper, let $d \ge 3$ be an integer and $\k$ be a field  of characteristic $0$ or $>d.$

In 1898, Hurwitz \cite{H} proved the well-known ``$1$, $2$,  $4$, $8$ Theorem" which states that the only values of $n$ for which there is an identity of the type
\begin{equation} \label{sq}
 (x_1^2 + x_2^2 + \cdots + x_n^2) (y_1^2 + y_2^2 + \cdots + y_n^2) =  z_1^2 + z_2^2 + \cdots + z_n^2
 \end{equation}
 where the $z_k$ are bilinear forms in $x=(x_1, x_2, \dots, x_n)$ and $y=(y_1, y_2, \dots, y_n)$ are $n = 1,2,4,8.$ Twenty years later, Hurwitz \cite{H1} and Radon\cite{Rad} proved independently that there is an identity
 \begin{equation} \label{hr}
 (x_1^2 + x_2^2 + \cdots + x_r^2) (y_1^2 + y_2^2 + \cdots + y_n^2) =  z_1^2 + z_2^2 + \cdots + z_n^2
 \end{equation}
 satisfying the previous conditions if and only if $r \le \rho(n)$ where $\rho(n)$ is the Hurwitz-Radon function defined as $\rho(n)=8a+2^b$ if $n=2^{4a+b}n_0$ where $n_0$ is odd and $0 \le b \le 3.$ In the 1960s, Pfister \cite{p} showed that if one allows $z_k$ to be linear in $y$ with coefficients in $\k(x)=\k(x_1, x_2, \dots, x_r),$ then there is an identity of form \eqref{sq} if and only if $n=2^m$  is a power of $2.$ These beautiful square identities are nice motivation for the algebraic theory of quadratic forms, and have close connection to various topics in algebra, arithmetics, combinatorics, geometry and topology.

For a long time, people are wondering if there are similar identities of products of sums of powers in higher degrees. For example, in 1975 Nathanson \cite{n} considered the following identity
\begin{equation} \label{np}
(x_1^d + x_2^d + \cdots + x_n^d) (y_1^d + y_2^d + \cdots + y_n^d) =  z_1^d + z_2^d + \cdots + z_n^d
\end{equation}
requiring that the $z_k$ are polynomials in $x$ and $y$ with integer coefficients. He applied simple arithmetic technique to show that such identities cannot hold for $(n, d)=(3, 3)$ and $n=2, d>2.$ He also conjectured that there exist no identities like \eqref{np} for any $n>1$ and $d>2.$ In the 1980s, from his work on higher pythagoras numbers, Becker \cite{becker} proposed a much stronger conjecture than the one of Nathanson's: if the following identity
\begin{equation}\label{bc}
(x_1^d + x_2^d + \cdots + x_r^d) (y_1^d + y_2^d + \cdots + y_s^d) =  z_1^d + z_2^d + \cdots + z_n^d
\end{equation}
holds, where each $z_k \in \k(x,y),$ then $rs \le n.$ To the best of our knowledge, so far very little seems to be known about this conjecture, but Gude made some progress in his thesis \cite{gude}. See also Shapiro's book \cite[Chapter 16]{shapiro} for a brief summary of Gude's results.

In this paper, we are mainly concerned about identities of the following type
\begin{equation} \label{hhr}
(x_1^d + x_2^d + \cdots + x_r^d) (y_1^d + y_2^d + \cdots + y_n^d) =  z_1^d + z_2^d + \cdots + z_n^d
\end{equation}
where each $z_k$ is a linear form in $y$ with coefficients in $\k(x).$ These identities are higher degree analogue of the well-known composition formulas of sums of squares of Hurwitz, Radon and Pfister. We show that such composition identities of sums of higher degree powers are trivial, namely, if \eqref{hhr} holds for $d>2,$ then $r=1.$ In particular, it follows that there are no nontrivial higher degree $n$-power formulas like \eqref{np}. As a matter of fact, we confirm a stronger version of the aforementioned Nathanson's conjecture. Our proof is simple and elementary, in which the crux is Harrison's center theory of higher degree forms \cite{h1}. We remark that nonexistence of \eqref{hhr} for $r>1, d>2$ is one of Gude's main results and his proof involves some sophisticated methods such as $p$-adic analysis, Jordan's theorem, Hurwitz curves, Kodaira dimension, etc. By the same idea, we also prove that if $r>1$ and $(n,d) \ne (2,3),$ then there is no identity of type
\begin{equation} \label{223}
(x_1^d + x_2^d + \cdots + x_r^d) (y_1^d + y_2^d + \cdots + y_n^d) =  z_1^d + z_2^d + \cdots + z_{n+1}^d
\end{equation}
where each $z_k$ is a linear form in $y$ with coefficients in $\k(x).$ This confirms some special case, namely $r=n=2$ and $d>3,$ of Becker's conjecture.

The notion of centers of homogeneous polynomials was introduced in the 1970s by Harrison in his seminal work \cite{h1} of the algebraic theory of higher degree forms. The centers of homogeneous polynomials were rediscovered independently by several other authors \cite{cg, k, pr, r} and found various applications in commutative algebra, algebraic geometry, multilinear algebra and invariant theory, see e.g. \cite{f1, hlyz, hlyz2, hlyz3}. One major application of centers is to the problem of direct sum decomposition, or in plain words, separation of variables of polynomials. There is one-to-one correspondence between the direct sum decompositions of a homogeneous polynomial of degree $>2$ and complete sets of orthogonal idempotents of the center algebra of the polynomial. Moreover, the direct sum decomposition into indecomposable summands for a higher degree homogeneous polynomial is essentially unique. This is the crucial ingredient used in our argument.

The paper is organized as follows. In Section 2 we recall Harrison's theory of centers and its application to direct sum decompositions of homogeneous polynomials. Then we apply the center theory to prove our main results in Section 3.

\section{Harrison centers of homogeneous polynomials}

Let $f(x_1, x_2, \dots, x_n) \in \k[x_1, \dots, x_n]$ be a homogeneous polynomial, also called form, of degree $d.$ Harrison \cite{h1} defined the center of $f=f(x_1, x_2, \dots, x_n)$ as
$$Z(f):=\{ X \in \k^{n \times n} \mid HX=X^TH \},$$ where $H=(\frac{\partial^2 f}{\partial x_i \partial x_j})_{1 \le i,\ j \le n}$ is the hessian matrix.

A homogeneous polynomial $f$ is said to be nondegenerate, if no variable can be removed by an invertible linear change of variables. It is not hard to observe that $f$ is nondegenerate if and only if the first-order differentials $\frac{\partial f}{\partial x_i}$ are linearly independent. To avoid triviality, usually it is enough to consider nondegenerate homogeneous polynomials. A homogeneous polynomial $f(x_1, x_2, \dots, x_n)$ is called a direct sum if, after an invertible linear change of variables $y=Px,$ it can be written as a sum of $t \ge 2$ forms in disjoint sets of variables as follows
\begin{equation*} \label{eds}
f=f_1(y_1, \dots, y_{a_1}) + f_2(y_{a_1+1},
\dots, y_{a_2}) + \cdots + f_t(y_{a_{t-1}+1}, \dots, y_n).
\end{equation*}
If this is not the case, then $f$ is said to be indecomposable. On the other extreme, if the $f_i$'s are forms in only one variable, then $f$ is said to be diagonalizable.

Corresponding to the homogeneous polynomial $f(x_1, x_2, \dots, x_n),$ there is an associated symmetric $d$-linear space as follows. First write $f$ in the symmetric way
\begin{equation*} \label{et}
f(x_1, \dots, x_n) = \sum_{1 \le i_1, \cdots, i_d \le n} a_{i_1 \cdots i_d} x_{i_1} \cdots x_{i_d}
\end{equation*}
where the $a_{i_1 \cdots i_d}$'s are symmetric, i.e., they remain unchanged under any permutation of their subscripts. Then take a vector space $V$ over $\k$ of dimension $n$ with a basis $e_1, \dots, e_n$ and define a symmetric $d$-linear function $\Theta \colon V \times \cdots \times V \longrightarrow \k$ by $\Theta(e_{i_1}, \dots, e_{i_d})=a_{i_1 \cdots i_d}.$ The pair $(V, \Theta)$ is called the associated symmetric $d$-linear space of $f$ under the basis $e_1, \dots, e_n.$ One can recover the polynomial $f$ from $(V, \Theta)$ as \[ f(x_1, \dots, x_n) = \Theta\left(\sum_{1 \le i \le n}x_ie_i, \ \dots, \ \sum_{1 \le i \le n}x_i e_i\right).\] Clearly, a change of basis for $V$ results in a change of variables for $f.$
The center can also be equivalently defined in terms of symmetric $d$-linear space $(V,\Theta)$ as
 \begin{equation*}
Z(V,\Theta):=\{ \phi \in \operatorname{End}(V) \mid \Theta(\phi(v_1), v_2, \dots, v_d)=\Theta(v_1, \phi(v_2), \dots, v_d), \ \forall v_1, v_2, \dots, v_d \in V \}.
\end{equation*}

Nonzero subspaces $V_1, \dots, V_t$ of $(V,\Theta)$ of a symmetric $d$-linear space $(V,\Theta)$ are said to be orthogonal, if $\Theta(v_1, \dots, v_d)=0$ for all $v_1, \dots, v_d \in V_1 \cup \cdots \cup V_t$ unless all the $v_i$'s are in the same $V_s$ for some $1 \le s \le t.$ If $V=V_1 \oplus \cdots \oplus V_t$ for $t \ge 2$ nonzero orthogonal subspaces, then call $(V,\Theta)$ decomposable. Otherwise, call $(V,\Theta)$ indecomposable. Clearly, the orthogonal decompositions of $(V,\Theta)$ are in bijection with the direct sum decompositions of its associated homogeneous polynomial $f.$ Suppose $V=V_1 \oplus \cdots \oplus V_t$ is an orthogonal decomposition. Choose a basis $\epsilon_1, \dots, \epsilon_{a_1}$ for $V_1,$ a basis $\epsilon_{a_1+1}, \dots, \epsilon_{a_2}$ for $V_2,$ and so on and so forth. Assume $(e_1, \dots, e_n)=(\epsilon_1, \dots, \epsilon_n)P$ with $P \in \operatorname{GL}_n(\k).$ Take the change of variables $x=Py,$ then we have the direct sum decomposition for the polynomial $f=f_1(y_1, \dots, y_{a_1}) + f_2(y_{a_1+1}, \dots, y_{a_2}) + \cdots + f_t(y_{a_{t-1}+1}, \dots, y_n).$ In fact, $f_i$ is the associated polynomials of $(V_i, \Theta|_{V_i}).$ More precisely,
\begin{equation} \label{dy}
f_i=\Theta\left(\sum_{a_{i-1}+1 \le i \le a_i}y_i\epsilon_i, \ \dots, \ \sum_{a_{i-1}+1 \le i \le a_i} y_i\epsilon_i \right).
\end{equation}
Conversely, given a direct sum decomposition $f=f_1(y_1, \dots, y_{a_1}) + f_2(y_{a_1+1}, \dots, y_{a_2}) + \cdots + f_t(y_{a_{t-1}+1}, \dots, y_n)$ with change of variables $x=Py,$ let $(\epsilon_1, \dots, \epsilon_n)=(e_1, \dots, e_n)P.$ Then let $V_1$ be the subspace spanned by $\epsilon_1, \dots, \epsilon_{a_1},$ $V_2$ be the subspace spanned by $\epsilon_{a_1+1}, \dots, \epsilon_{a_2},$ etc. It is clear that $V=V_1 \oplus \cdots \oplus V_t$ is an orthogonal decomposition.
 
The following are some useful facts about centers and direct sum decompositions of homogeneous polynomials obtained in \cite{h1}, see also \cite{h2, hlyz, hlyz2, hlyz3}.

\begin{proposition} \label{pc}
Suppose $f \in \k[x_1, x_2, \dots, x_n]$ is nondegenerate and is of degree $d.$ Then
\begin{itemize}
\item[(1)] The center $Z(f)$ is a commutative subalgebra of the full matrix algebra $\k^{n \times n}.$
\item[(2)] If $f=f_1+f_2+\cdots+f_t$ is a direct sum decomposition, then $Z(f) \cong Z(f_1) \times Z(f_2) \times \cdots \times Z(f_t).$
\item[(3)] There is a one-to-one correspondence between direct sum decompositions of $f$ and complete sets of orthogonal idempotents of $Z(f).$
\item[(4)] If $f=g_1+g_2+\cdots+g_s$ and $f=h_1+h_2+\cdots+h_t$ are direct sum decompositions, where the $g_i$ and $h_j$ are indecomposable, then $s=t$ and after a reordering $g_i=h_i$ for all $1 \le i \le s.$
\item[(5)] $f$ is diagonalizable if and only if $Z(f) \cong \k \times \k \times \cdots \times \k \ ($n$ \ copies).$
\end{itemize}
\end{proposition}

\begin{proof}
For the convenience of the reader, we include a sketchy proof for item (4) which is the key ingredient used to prove our main results. More details can be found in the aforementioned references. Let $(V, \Theta)$ be the associated symmetric $d$-linear space of $f.$ Suppose $f=g_1+g_2+\cdots+g_s=h_1+h_2+\cdots+h_t$ as in the proposition. Then there are two associated complete sets of primitive orthogonal idempotents in $Z(V, \Theta)$ corresponding to the given two direct sum decompositions of $f.$ However, the center algebra $Z(V, \Theta)$ is finite-dimensional and commutative, so it has only one unique complete set of primitive orthogonal idempotents \cite{hlyz3}. Thus for $V,$ there is a unique orthogonal decomposition $V=V_1 \oplus \cdots \oplus V_r$ where each $V_i$ is indecomposable. As \eqref{dy}, each $g_i$ or $h_j$ is the associated polynomial of some $(V_k, \Theta|_{V_k}).$ It follows directly that $s=t=r$ and after a reordering $g_i=h_i$ for all $1 \le i \le s.$
\end{proof}

\section{Nonexistence of some products of sums of powers}
In this section, we apply Harrison's center theory, in particular the uniqueness of the direct sum decomposition of indecomposable summands, to prove the nonexistence of some products of sums of powers.

\begin{theorem} \label{mt1}
Suppose $d>2.$ There is an identity of product of sums of powers
\[ (x_1^d + x_2^d + \cdots + x_r^d) (y_1^d + y_2^d + \cdots + y_n^d) =  z_1^d + z_2^d + \cdots + z_n^d  \]
where each $z_k$ is a linear form in $y$ with coefficients in $\k(x)$ if and only if $r=1.$
\end{theorem}

\begin{proof}
View both sides of the previous identity as polynomials in $y$ with coefficients in $\k(x).$ First note that if the $z_k$ are linearly dependent as linear forms in $y,$ then the polynomial $z_1^d + z_2^d + \cdots + z_n^d$ is degenerate in $\k(x)[y]$ as it can be presented by fewer variables, namely a maximally linearly independent subset of the $z_k.$ Now assume the $z_k$ are linearly independent. Then $z_1^d + z_2^d + \cdots + z_n^d$ is a direct sum, and obviously each $z_k^d$ is indecomposable. Thus by item (4) of Proposition \ref{pc}, after a reordering we may assume $(x_1^d + x_2^d + \cdots + x_r^d)y_j^d=z_j^d$ for all $1 \le j \le n.$ By assumption $z_k \in \k(x)[y],$ hence $z_k=\frac{p(x)}{q(x)}y_k$ where $p(x), q(x) \in \k[x].$ It follows that $x_1^d + x_2^d + \cdots + x_r^d=\frac{p(x)^d}{q(x)^d}.$ Since $\k[x]$ is a unique factorization domain, this forces $q(x) | p(x).$ Therefore we can assume in the beginning $q(x)=1.$ Now we have $x_1^d + x_2^d + \cdots + x_r^d=p(x)^d.$ This means $p(x)$ is a $\k$-linear form in $x.$ Obviously, this happens if and only if $r=1.$
\end{proof}

\begin{remark}
\emph{It follows easily by the previous theorem that there are no nontrivial identities like \eqref{np} even if the $z_k$ are allowed to be chosen from $\k(x)[y].$ Therefore, we confirm a stronger version of Nathason's conjecture.}
\end{remark}

By the same idea, Theorem \ref{mt1} can be slightly generalized to the following

\begin{theorem} \label{mt2}
Suppose $r>1, \ n>1$ and $(n,d) \neq (2,3).$ Then there is no identity of the following type
\[ (x_1^d + x_2^d + \cdots + x_r^d) (y_1^d + y_2^d + \cdots + y_n^d) =  z_1^d + z_2^d + \cdots + z_{n+1}^d  \]
where each $z_k$ is a linear form in $y$ with coefficients in $\k(x).$
\end{theorem}

\begin{proof}
Assume the contrary and we will seek for a contradiction. Just as in the proof of the previous theorem, we may assume without loss of generality that $z_1, z_2, \dots, z_n$ are linearly independent and $z_{n+1}=\sum_{1 \le k \le n} a_k z_k$ with $a_k \in \k(x).$ If some $a_k$ are zero, say $a_l=\cdots=a_n=0,$ then it is clear that $z_n$ is an indecomposable summand and so $z_n^d=(x_1^d + x_2^d + \cdots + x_r^d) y_j^d$ for some $j.$ For the same reason as in Theorem \ref{mt1}, this is impossible. If $a_1a_2 \cdots a_n \ne 0,$ then we compute the center of $g=z_1^d + z_2^d + \cdots + z_{n+1}^d$ and come up with a contradiction. Indeed, let $H$ be the hessian matrix of $g,$ then $H_{ii}=d(d-1)(z_i^{d-2}+a_i^2z_{n+1}^{d-2})$ for all $1 \le i \le n$ and $H_{ij}=d(d-1)a_ia_jz_{n+1}^{d-2}$ for all $i \ne j$. Suppose $X=(x_{ij})\in Z(g).$ According to the definition of centers, compare the $ij$-entry and $ji$-entry of $HX$ for all $i \ne j,$ then we have
\[ z_{i}^{d-2}x_{ij} + z_{n+1}^{d-2} \sum_{1 \le k \le n}a_ia_kx_{kj} = z_{j}^{d-2}x_{ji} + z_{n+1}^{d-2} \sum_{1 \le k \le n}a_ja_kx_{ki}. \] It is not hard to observe that if $(n,d) \ne (2,3),$ then $z_{i}^{d-2}, z_{j}^{d-2}, z_{n+1}^{d-2}$ are linearly independent. It follows easily that $x_{ij}=0$ and $x_{ii}=x_{jj}$ for all $i \ne j,$ that is, $Z(g) \cong \k(x).$ This is absurd, since the center of $(x_1^d + x_2^d + \cdots + x_r^d) (y_1^d + y_2^d + \cdots + y_n^d)$ is isomorphic to $\k(x)^n$ by item (5) of Proposition \ref{pc}. The proof is completed.
\end{proof}

\begin{remark}
\emph{If $(n,d)=(2,3),$ then $z_1, z_2, z_3$ are obviously linearly dependent. Thus the previous argument does not work in this situation. For the moment, we do not know if one can deal with this special case by the theory of centers.}
\end{remark}

\begin{corollary}
Suppose $d>3.$ There holds an identity
\begin{equation} \label{22n}
(x_1^d + x_2^d ) (y_1^d + y_2^d ) =  z_1^d + z_2^d + \cdots +z_n^d
\end{equation}
where each $z_k$ is a linear form in $y$ with coefficients in $\k(x)$ if and only if $n>3.$
\end{corollary}

\begin{proof}
Consider the special case $r=n=2$ of Theorem \ref{mt2}. It follows that the identity \eqref{22n} holds only if $n>3.$ Conversely, if $n>3,$ let $z_1=x_1y_1, \ z_2=x_1y_2, \ z_3=x_2y_1, \ z_4=x_2y_2$ and the other $z_k=0,$ then obviously \ref{22n} holds.
\end{proof}

\begin{remark}
\emph{This corollary confirms a special case of Becker's conjecture, namely the case of $r=s=2$ and $d>3$ with $z_k$ assumed to be polynomials in $y.$}

\end{remark}

\end{document}